\documentclass[11pt,a4paper]{amsart}

\usepackage{amscd,amssymb,amsopn,amsmath,amsthm,mathrsfs,graphics,amsfonts,enumerate,verbatim,calc}

\usepackage{amssymb,amsmath}

\headsep=1cm \numberwithin{equation}{section}
\newtheorem{theorem}{Theorem}[section]
\newtheorem{lemma}[theorem]{Lemma}
\newtheorem{proposition}[theorem]{Proposition}
\newtheorem{corollary}[theorem]{Corollary}

\newtheorem{problem}{Problem}
\newtheorem{maintheorem}{Main Theorem}

\theoremstyle{definition}
\newtheorem{definition}[theorem]{Definition}
\theoremstyle{remark}
\newtheorem{remark}[theorem]{Remark}

\newtheorem{acknowledgement}{Acknowledgement}

\newcommand{\coker}{\operatorname{coker}}

\newcommand{\fm}{\frak{m}}

\newcommand{\fn}{\frak{n}}

\begin{document}
\title[Integral perfectoid big Cohen-Macaulay algebras via Andr\'e'S theorem]
{Integral perfectoid big Cohen-Macaulay algebras via Andr\'e's theorem}

\author[K.Shimomoto]{Kazuma Shimomoto}
\email{shimomotokazuma@gmail.com}

\thanks{2010 {\em Mathematics Subject Classification\/}: 11S15, 13A35, 13B22, 13B40, 13F35, 13H10, 13J10}

\keywords{Big Cohen-Macaulay algebra, Fontaine ring, Frobenius map, Witt vectors}


\begin{abstract}
The main result of this article is to prove that any Noetherian local domain of mixed characteristic maps to an integral perfectoid big Cohen-Macaulay algebra. The proof of this result is based on the construction of almost Cohen-Macaulay algebras in mixed characteristic due to Yves Andr\'e. Moreover, we prove that the absolute integral closure of a complete Noetherian local domain of mixed characteristic maps to an integral perfectoid big Cohen-Macaulay algebra.
\end{abstract}

\maketitle 

\tableofcontents

\section{Introduction}

Let $(R,\fm)$ be a Noetherian local ring and let $M$ be an $R$-module. Then $M$ is a \textit{big Cohen-Macaulay R-module}, if $M \ne \fm M$ and there is a system of parameters $\mathbf{x}=x_1,\ldots,x_d$ of $R$ such that $\mathbf{x}$ is a regular sequence on $M$. To be precise, one should call $M$ a big Cohen-Macaulay $R$-module with respect to $\mathbf{x}$. The notion of big Cohen-Macaulay modules was introduced by Hochster in 1970's for the purpose of studying the homological conjectures (see \cite{H07} for a brief survey and \cite{HH95} in the equal characteristic case). Hochster proved that any local ring $(R,\fm)$ of \textit{equal} characteristic admits a big Cohen-Macaulay module. The issue whether such a module exists in mixed characteristic has remained unclear for a long time. Quite recently, Andr\'e proved that big Cohen-Macaulay algebras exist in the mixed characteristic case in \cite{An1} and \cite{An2}, which also implies that the Direct Summand Conjecture is fully settled (see \cite{B16} for the derived variant of this conjecture).

\begin{theorem}[Y. Andr\'e]
\label{MainAndre}
Every complete Noetherian local domain of mixed characteristic admits a big Cohen-Macaulay algebra.
\end{theorem}

This theorem can be deduced from Theorem \ref{Andre}, which we explain below. The primary aim of this article is to show abundance of big Cohen-Macaulay algebras with certain distinguished properties based on Theorem \ref{Andre} stated below. First, we state the main theorem (see Theorem \ref{integralperfect}).

\begin{maintheorem}
\label{main}
Let $(R,\fm)$ be a Noetherian local domain of mixed characteristic. Then there exists an $R$-algebra $T$ such that $T$ is an integral perfectoid big Cohen-Macaulay $R$-algebra. 

Moreover, we have the following assertions:
\begin{enumerate}
\item
Assume that $R$ is a complete Noetherian local domain of mixed characteristic with perfect residue field. Then there exists an integral perfectoid big Cohen-Macaulay $R$-algebra $T$ with the property that $R \to T$ factors as $R \to S \to T$, such that $S$ is an integral pre-perfectoid normal domain that is integral over $R$ and $R[\frac{1}{pg}] \to S[\frac{1}{pg}]$ is ind-\'etale for some nonzero element $g \in R$.

\item
Assume that $R$ is a complete Noetherian local domain of mixed characteristic. Let $B$ be an integral almost perfectoid, almost Cohen-Macaulay $R$-algebra such that $R \to B$ factors as $R \to S \to B$ and $S$ is an integral perfectoid algebra containing compatible systems of elements: $
\{p^{\frac{1}{p^n}}\}_{n  \ge 0},\{x_2^{\frac{1}{p^n}}\}_{n  \ge 0},\ldots,\{x_d^{\frac{1}{p^n}}\}_{n  \ge 0}$ for a system of parameters $p,x_2,\ldots,x_d$ of $R$. Then there is a ring homomorphism $S \to T$ such that $T$ is an integral perfectoid big Cohen-Macaulay $R$-algebra.
\end{enumerate}
\end{maintheorem}

One can indeed prove the following stronger result (see Corollary \ref{absoluteBig} and Remark \ref{completion}).

\begin{maintheorem}
Let $(R,\fm)$ be a Noetherian local domain of mixed characteristic and let $R^+$ be its absolute integral closure. Then $R^+$ maps to an integral perfectoid big Cohen-Macaulay $R$-algebra.
\end{maintheorem}

For basic notion and notations, we refer the reader to \S~\ref{notation}.
An \textit{integral perfectoid algebra} is a $p$-torsion free, $p$-adically complete algebras on which the Frobenius endomorphism is surjective modulo $p$ whose kernel is a principal ideal (see Definition \ref{preperfectoid}). See Definition \ref{DefAlmost} for the definition of \textit{almost Cohen-Macaulay algebras}. We point out that Main Theorem \ref{main} is somehow reminiscent of the classical result of Hochster and Huneke \cite{HH92}, which asserts that the absolute integral closure $R^+$ of an excellent local domain $(R,\fm)$ of equal characteristic $p>0$ is a big Cohen-Macaulay algebra, in which $R^+$ is an example of an integral perfectoid algebra in the positive characteristic case. We refer the reader to \cite{HL07}, \cite{Q16} and \cite{SS12} for the recent developments on the Cohen-Macaulayness of the absolute integral closure in positive characteristic. The utility of integral perfectoid algebras is contained in the fact that one can reduce the study of big Cohen-Macaulay algebras in mixed characteristic to that of big Cohen-Macaulay algebras in positive characteristic via tilting operations in favorable circumstances (see Corollary \ref{corFontaine} for more on this). 

The proof of Main Theorem \ref{main} relies heavily on the following remarkable result of Andr\'e, together with the construction of a certain seed algebra over the Fontaine ring, using the Frobenius map.

\begin{theorem}[Y. Andr\'e]
\label{Andre}
Every complete Noetherian local domain of mixed characteristic admits an integral almost perfectoid, almost Cohen-Macaulay algebra.
\end{theorem}

Thus, Main Theorem \ref{main} is regarded as a refinement of Theorem \ref{Andre}, which is found in \cite{An2} and its proof uses a deep theorem, \textit{Perfectoid Abhyankar's Lemma} as proved in \cite{An1}. The existence of big Cohen-Macaulay algebras can be deduced from Theorem \ref{Andre} by combining Hochster's method of \textit{(partial) algebra modifications}. The method of (partial) algebra modifications works in a quite flexible manner for rings in positive characteristic, so we need to pass to the positive characteristic via Fontaine rings and then use the results of Dietz \cite{Di07}. We recall that an algebra $T$ over a local Noetherian ring is a \textit{seed}, if it maps to a big Cohen-Macaulay algebra (see Definition \ref{seedalg}). Now let us summarize the key tools that we need in this article.

\begin{enumerate}
\item[$\bullet$]
Andr\'e's theorem on existence of an integral almost perfectoid almost Cohen-Macaulay algebra in mixed characteristic (see Theorem \ref{almostCM} for the precise statement).

\item[$\bullet$]
By the standard method, we are reduced to the assumption that $(R,\fm)$ is a complete Noetherian local normal domain of mixed characteristic with perfect residue field in Main Theorem \ref{main} (see the beginning of the proof of Theorem \ref{integralperfect}).

\item[$\bullet$]
We make use of Fontaine rings and their Witt vectors. In the language of perfectoid spaces, these correspond to the operations \textit{tilting} and \textit{untilting} in perfectoid geometry.

\item[$\bullet$]
A seed over a Noetherian local ring of equal characteristic $p>0$ maps to an absolutely integrally closed, big Cohen-Macaulay $R$-algebra domain. This is due to Dietz (see Theorem \ref{seed}). 
\end{enumerate}

\section{Notation and conventions}
\label{notation}

All rings are assumed to be commutative and unitary. A \textit{quasilocal ring} is a commutative ring with a unique 
maximal ideal. A \textit{local ring} is a Noetherian commutative ring $R$ with a unique maximal ideal $\fm$, which is denoted by $(R,\fm)$. Let $p>0$ be a prime number and let $A$ be a ring. We denote by $A^{\wedge}:=\varprojlim_{n>0} A/p^nA$ the $p$-adic completion of $A$. Let $(R,\fm)$ be a local Noetherian ring and let $M$ be an $R$-module. Then we say that $M$ is a \textit{big Cohen-Macaulay R-module}, if $\fm M \ne M$ and there is a system of parameters of $R$ that is a regular sequence on $M$. A big Cohen-Macaulay $R$-module $M$ is \textit{balanced}, if every system of parameters of $R$ is a regular sequence on $M$. Later, we shall recall the definition of \textit{almost Cohen-Macaulay algebras} following \cite{An2}. For any ring $B$, let $W(B)$ denote the ring of $p$-typical Witt vectors. When there is no danger of confusion, we omit to add ``$p$-typical'' to the Witt vectors. The book \cite{S73} is a standard reference for the Witt vectors. Let $A$ be an $\mathbb{F}_p$-algebra for $\mathbb{F}_p:=\mathbb{Z}/p\mathbb{Z}$ and a prime number $p>0$. Then we say that $A$ is a \textit{perfect $\mathbb{F}_p$-algebra}, if the Frobenius endomorphism $F_A:A \to A~(F_A(a)=a^p)$ is bijective. Let $A$ be an integral domain. We define the \textit{absolute integral closure} of $A$, denoted by $A^+$, to be the integral closure of $A$ in a fixed algebraic closure of the field of fractions of $A$ (see \cite{Ar71} for details). Fix a prime number $p>0$. A \textit{compatible system of elements} is a sequence of elements $\{x_n\}_{n \ge 0}$ in a ring $A$ such that $x_{n+1}^p=x_n$ for all $n \ge 0$.

\section{Fontaine rings and the Witt vectors}

We follow the article \cite{Sh11} for the treatment of Fontaine ring associated to a ring.

\begin{definition}[Fontaine ring]
Let $A$ be a commutative ring and let $p>0$ be a prime number. The \textit{Fontaine ring} (or \textit{tilt})\footnote{"Tilt" is a standard name now. In \cite{Sh11}, "Fontaine ring" was employed and $\mathbf{E}(A)$ was used for the Fontaine ring. As this symbol is obsolete now, we decided to switch to a standard one.} associated to $A$ is defined as the inverse limit:
$$
A^\flat:=\varprojlim_{n \in \mathbb{N}} A_n,
$$
where $A_n:=A/pA$ for all $n \in \mathbb{N}$ and the transition map $A_{n+1} \to A_n$ is defined by $x \mapsto x^p$ (the Frobenius endomorphism).
\end{definition}

It is easy to check from the definition that $A^\flat$ is a nonzero perfect $\mathbb{F}_p$-algebra if $A/pA \ne 0$. An element in the Fontaine ring $A^\flat$ can be written as a vector $(a_0,a_1,a_2,\ldots)$ such that $a_n \in A_{n+1}$ and $a_{n+1}^p=a_n$. Fix an element $a \in A$. Then we often use the symbol $a^\flat:=(a_0,a_1,a_2,\ldots) \in A^\flat$ such that $a_0$ is the image of $a$ in $A/pA$, when a specific choice of the sequence $\{a_n\}_{n \ge 0}$ is not important or does not cause any confusion. For example, if $p^{\frac{1}{p^m}} \in A$ for $m>0$, we write
$p^\flat:=(p,p^{\frac{1}{p}},\ldots) \in A^\flat$.\footnote{It is probably better to write $p^\flat:=(\overline{p},\overline{p^{\frac{1}{p}}},\ldots) \in A^\flat$, but we make a simple choice of notation.} Scholze introduced \textit{perfectoid K-algebras} in \cite{Sch12}, where $K$ is a perfectoid field. In this article, we deal with a version introduced by Bhatt-Morrow-Scholze in \cite[Definition 3.5]{BMS17}.\footnote{Fontaine, Gabber-Ramero and Kedlaya-Liu also considered versions of perfectoid algebras.} The monographs \cite{GR} and \cite{KL15} provide comprehensive information on perfectoid rings and spaces.

\begin{definition}
\label{preperfectoid}
Let $A$ is a $p$-torsion free algebra such that $A/pA \ne 0$.

\begin{enumerate}
\item
$A$ is called \textit{integral pre-perfectoid}, if there is an element $\varpi \in A$ such that $\varpi^p=p \cdot u$ for some unit $u \in A$ and the Frobenius endomorphism $A/pA \to A/pA$ is surjective with its kernel being equal to $\varpi A$. Moreover, an integral pre-perfectoid algebra is called \textit{integral perfectoid}, if $A$ is $p$-adically complete.

\item
Assume that there is an element $\varpi \in A$ such that $\varpi^p=p \cdot u$ for some unit $u \in A$. Put $\pi:=pg$ for a nonzero divisor $g \in A$ and assume that there are compatible systems of elements: $\{p^{\frac{1}{p^n}} \in A\}_{n \ge 0}$, $\{g^{\frac{1}{p^n}} \in A\}_{n \ge 0}$ and $\pi^{\frac{1}{p^\infty}}:=\{\pi_n \in A\}_{n \ge 0}$ in the sense that $\pi_{i+1}^p=\pi_i$ and $\pi_0:=\pi$. Then $A$ is called \textit{integral almost perfectoid with respect to $\pi^{\frac{1}{p^\infty}}$}, if $A$ is $p$-adically complete and the Frobenius endomorphism $A/pA \to A/pA$ induces an injection $A/\varpi A \hookrightarrow A/pA$ whose cokernel is annihilated by $\pi^{\frac{1}{p^n}}$ for all $n>0$.
\end{enumerate}
\end{definition}

The reader should notice the following embedding of categories:
$$
\biggl\{\mbox{Integral perfectoid algebras}\biggl\} \to \biggl\{\mbox{Integral pre-perfectoid algebras}\biggl\},
$$
which is fully faithful and its left inverse is given by taking the $p$-adic completion. Let us collect the basic results on $A^\flat$ (see \cite[Proposition 4.5, Lemma 4.6 and Proposition 4.7]{Sh11} for references). They are stated only for the absolutely integrally closed domains in \cite{Sh11}. However, all of these results can be extended to integral pre-perfectoid algebras without any change by adding one assumption that $p^{\frac{1}{p^m}} \in A$ for all $m>0$.\footnote{Precisely speaking, this is not necessary in the general theory of perfectoid geometry. But this extra assumption will be sufficient for the main results in this paper.} The most comprehensive reference that covers these results is \cite{GR}. See \cite[Proposition 4.5]{Sh11} for the following proposition. Fix a prime number $p>0$.

\begin{proposition}
\label{FontaineRing}
Assume that $A$ is an integral pre-perfectoid algebra such that $p^{\frac{1}{p^m}} \in A$ for all $m>0$. Then the following statements hold:
\begin{enumerate}

\item
$A^\flat$ is a perfect $\mathbb{F}_p$-algebra and it is complete in the $p^\flat$-adic topology, where $p^\flat:=(p,p^{\frac{1}{p}},\ldots) \in A^\flat$.

\item
There is a natural ring surjection: $\Phi_A:A^\flat \to A/pA$ defined by the rule $(a_0,a_1,a_2,\ldots) \mapsto a_0$.
Moreover, there is a short exact sequence:
$$
0 \to A^\flat \xrightarrow{\times p^\flat} A^\flat \xrightarrow{\Phi_A} A/pA \to 0.
$$
In particular, the element $p^\flat$ is a nonzero divisor of $A^\flat$.
\end{enumerate}
\end{proposition}

Let $A$ be an integral pre-perfectoid algebra and let $A^{\wedge}$ be the $p$-adic completion of $A$. Then we have $A^{\wedge}/pA^{\wedge} \cong A/pA$. In order to recover the integral perfectoid algebra $A^{\wedge}$ from $A^\flat$, we need the following lemma (see \cite[Lemma 4.6]{Sh11}).

\begin{lemma}
\label{Witt}
Assume that $A$ is an integral pre-perfectoid algebra such that $p^{\frac{1}{p^m}} \in A$ for all $m>0$. Then there is a commutative diagram:
$$
\begin{CD}
A^\flat @>\widehat{\Phi}_{A}>> A^{\wedge} \\
@| @V \pi VV \\
A^\flat @>\Phi_{A}>> A/pA
\end{CD}
$$
in which $\pi$ is the natural projection and $\widehat{\Phi}_{A}$ is uniquely determined such that $\widehat{\Phi}_{A}$ is a multiplicative map and $\widehat{\Phi}_{A}(1^\flat)=1$. In fact, suppose that $A^{\wedge}$ contains a compatible system of elements $\{x^{\frac{1}{p^n}}\}_{n \ge 0}$ for an element $x \in A^{\wedge}$. Then we have $\widehat{\Phi}_{A}(x^\flat)=x$ for $x^\flat=(x,x^{\frac{1}{p}},x^{\frac{1}{p^2}},\ldots) \in A^\flat$.
\end{lemma}

Let $W(A^\flat)$ be the ring of $p$-typical Witt vectors of $A^\flat$. Notice that since $A^\flat$ is a perfect $\mathbb{F}_p$-algebra, $W(A^\flat)$ can be characterized by the property that it is a unique $p$-torsion free, $p$-adically complete algebra such that $W(A^\flat)/p W(A^\flat) \cong A^\flat$. If $A \to B$ is a homomorphism of integral pre-perfectoid algebras, we get the following functorial commutative diagram:
$$
\begin{CD}
W(A^\flat) @>>> W(B^\flat) \\
@VVV @VVV \\
A^\flat @>>> B^\flat \\
@V\Phi_A VV @V\Phi_B VV \\
A/pA @>>> B/pB \\
\end{CD}
$$

Let us define the map $\theta_{A^\flat}:A^\flat \to W(A^\flat)$ by letting $\theta_{A^\flat}(a)=(a,0,\ldots,0,\ldots)$. This map is multiplicative, but not additive. This map is called the \textit{Teichm$\ddot{u}$ller lift}. Pick a point $\underline{a}=(a_{0},\ldots,a_{n},\ldots) \in W(A^\flat)$. Since $A^\flat$ is a perfect $\mathbb{F}_p$-algebra, the equation $x^{p^n}=a$ has a root $a^{p^{-n}}$ for any element $a \in A^\flat$. We define the map:
$$
\psi:W(A^\flat) \to A^{\wedge}
$$
by the rule 
$$
\psi(\underline{a})=\sum_{n=0}^{\infty}p^{n} \cdot \widehat{\Phi}_{A}(a_{n}^{p^{-n}}),
$$
in which the right hand side makes sense in $A^{\wedge}$. It is easy to check that $\psi \circ \theta_{A^\flat}=\widehat{\Phi}_{A}$, where $\theta_{A^\flat}$ is as above. The following proposition claims that $\psi$ defines a surjective ring homomorphism (see \cite[Proposition 4.7]{Sh11}).

\begin{proposition}
\label{Teichmuller}
Assume that $A$ is an integral pre-perfectoid algebra such that $p^{\frac{1}{p^m}} \in A$ for all $m>0$. Then the following statements hold.
\begin{enumerate}
\item
$\psi:W(A^\flat) \to A^{\wedge}$ is a ring homomorphism.

\item
There is a short exact sequence:
$$
0 \to W(A^\flat) \xrightarrow{\times \vartheta} W(A^\flat) \xrightarrow{\psi} A^{\wedge} \to 0 \\
$$
for $\vartheta:=\theta_{A^\flat}(p^\flat)-p$. Moreover, $(\vartheta, p)$ (resp. $(p,\varphi)$) forms a regular sequence on $W(A^\flat)$ and there is a natural ring homomorphism:
$$
A \to A^{\wedge} \cong \frac{W(A^\flat)}{\vartheta \cdot W(A^\flat)}.
$$
\end{enumerate}
\end{proposition}

The last assertion in Proposition \ref{Teichmuller} states that $A^{\wedge}$ can be recovered from the Fontaine ring $A^\flat$, if $A$ is an integral pre-perfectoid algebra in a functorial manner.

\section{Big Cohen-Macaulay algebras and seeds in characteristic $p>0$}

In order to construct a balanced big Cohen-Macaulay module from a big Cohen-Macaulay module, the following result is quite useful.

\begin{proposition}
\label{Strooker}
Let $(R,\fm)$ be a Noetherian local ring and let $M$ be a big Cohen-Macaulay $R$-module and let $\widehat{M}$ denote the $\fm$-adic completion of $M$. Then $\widehat{M}$ is a balanced big Cohen-Macaulay $R$-module.
\end{proposition}

\begin{proof}
The proof is found in \cite[Theorem 1.7]{BS83}.
\end{proof}

The notion of a seed over a Noetherian local ring was introduced and studied extensively by G. Dietz (see \cite[Definition 3.1]{Di07}). Notice that the definition is characteristic-free.

\begin{definition}
\label{seedalg}
Let $(R,\fm)$ be a Noetherian local ring and let $T$ be an $R$-algebra. Then we say that $T$ is a \textit{seed}, if $T$ maps to a big Cohen-Macaulay $R$-algebra.
\end{definition}

Let us recall the following crucial result from \cite{Di07}.

\begin{theorem}
\label{seed}
Let $(R,\fm)$ be a Noetherian local ring of equal characteristic $p>0$ and suppose that $T$ is a seed over $R$. Then the following hold:
\begin{enumerate}
\item
$T$ maps to a perfect $\fm$-adically complete balanced big Cohen-Macaulay $R$-algebra. 

\item
$T$ can be mapped to an absolutely integrally closed, $\fm$-adically separated, quasilocal balanced big Cohen-Macaulay $R$-algebra domain.
\end{enumerate}
\end{theorem}

\begin{proof}
The assertion $\rm(1)$ is due to \cite[Proposition 3.7]{Di07}. The assertion $\rm(2)$ is due to \cite[Theorem 7.8]{Di07}.
\end{proof}

\section{Almost Cohen-Macaulay algebras}

In this section, we construct a certain large integral extension of a complete Noetherian local domain in order to apply the constriction of Fontaine rings. To this aim, we need to consider the \textit{maximal \'etale extension} of a normal domain as discussed in \cite{Sh17}, which we recall briefly. Let $A \hookrightarrow B$ be a ring extension such that $A$ is a normal domain and $B$ is reduced and torsion free over $A$. Then the maximal \'etale extension of $A$ in $B$ is defined as the colimit of all $C$ such that $A \hookrightarrow C \hookrightarrow B$ and $C$ is finite \'etale over $A$. Denote this extension ring by $A_B^{\rm{\acute{e}t}}$, or just $A^{\rm{\acute{e}t}}$ if no confusion is likely to occur. We should keep in mind that
any ring $D$ between $A$ and $A^{\rm{\acute{e}t}}$ may fail the property that $D$ is ind-\'etale over $A$. The following lemma will be useful to construct an integral pre-perfectoid algebra from a $p$-torsion free algebra under mild conditions.

\begin{lemma}
\label{semiperfect}
Assume that $A$ is a $p$-torsion free normal domain such that $pA \ne A$ and $p$ is contained in any maximal ideal of $A$. Let $A \to A^+$ be a natural injection from $A$ into the absolute integral closure $A^+$. Let $B$ be the maximal \'etale extension of $A[p^{-1}]$ such that $B \subset A^+[p^{-1}]$, so that we have $A[p^{-1}] \subset B \subset A^+[p^{-1}]$. Take $\overline{A}$ to be the integral closure of $A$ inside $B$. 

Then $\overline{A}$ is a semiperfect ring. Moreover, there exists a sequence of elements $\pi_n \in \overline{A}$ such that $\pi_{n+1}^p=\pi_n$, $\pi_1^p=p$ and the Frobenius map on $\overline{A}/p\overline{A}$ induces an isomorphism: $\overline{A}/\pi \overline{A} \cong \overline{A}/p \overline{A}$.
\end{lemma}

\begin{proof}
This is found in \cite[Lemma 10.1]{Sh17} and we give its proof here, because it is concrete and contains useful ideas. 

$A$ is a normal domain and $A[p^{-1}] \to \overline{A}[p^{-1}]$
is ind-\'etale by assumption, so that $\overline{A}[p^{-1}]$ is normal, which also implies that $\overline{A}$ is normal. Note that $pA \ne A$ implies that $p$ is not a unit element of $A$. Under the stated assumption that $p$ is contained in any maximal ideal of $A$, it follows that $p$ is contained in the Jacobson radical of $\overline{A}$, because every maximal ideal of $\overline{A}$ lies over some maximal ideal of $A$. Pick an element $b \in \overline{A}$ and consider a polynomial
$$
f(X):=X^{p^2}-pX-b \in \overline{A}[X].
$$
Then $f'(X)=p^2X^{p^2-1}-p=p(pX^{p^2-1}-1)$ and the localization map:
$$
\overline{A}[p^{-1}] \to \overline{A}[X]/(f(X))[p^{-1}]
$$
is finite \'etale, because $p$ is contained in the Jacobson radical of $\overline{A}[X]/(f(X))$ and therefore, the image of $pX^{p^2-1}-1$ in $\overline{A}[X]/(f(X))$ is a unit element. There exist an element $a \in A^+$ such that $f(a)=0$ and a commutative diagram:
$$
\begin{CD}
\overline{A} @>>>  \overline{A}[X]/(f(X)) \\
@| @VVV \\
\overline{A} @>>>  \overline{A}[a] @>>> A^+ \\
\end{CD}
$$
where $\overline{A}[X]/(f(X)) \to \overline{A}[a]$ is defined by mapping $X$ to $a$. Then $\overline{A}[X]/(f(X))][p^{-1}]$ is isomorphic to a finite product of normal domains in view of the normality of $\overline{A}$, and $\overline{A}[a][p^{-1}]$ is isomorphic to one of the factors of $\overline{A}[X]/(f(X))][p^{-1}]$. This shows that 
$$
A[p^{-1}] \to  \overline{A}[a][p^{-1}]
$$
is ind-\'etale. As $A[p^{-1}] \to \overline{A}[p^{-1}]$ is the maximal \'etale extension in $A^+[p^{-1}]$, it follows that $a \in \overline{A}$. Finally, we have 
$$
a^{p^2}-b \equiv a^{p^2}-pa-b \equiv 0 \pmod{p \overline{A}}
$$ 
and $(a^p)^p \equiv b \pmod{p \overline{A}}$. This proves that the Frobenius map is surjective on $\overline{A}/p \overline{A}$. 

Let $\pi_n \in A^+$ be a root of the equation $X^{p^n}-p=0$ and let $A[\pi_n]$ be the subring of $A^+$. Then $\overline{A} \to \overline{A}[\pi_n]$ is \'etale after inverting $p$, so we have $\pi_n \in \overline{A}$. Put $\pi:=\pi_1$. To deduce an isomorphism $\overline{A}/\pi \overline{A}  \cong \overline{A}/p \overline{A}$, it suffices to show that the kernel of the Frobenius map:
$$
F: \overline{A}/p \overline{A} \to \overline{A}/p \overline{A}
$$
is principally generated by $\pi$. Assume that $\overline{x}^p=0$ for $\overline{x} \in \overline{A}/(p)$ with its lift $x \in \overline{A}$. Then we can write $x^p=p \cdot b$ for some $b \in \overline{A}$, which implies that $x=\pi \cdot b'$ with $b' \in A^+$ and
$$
b' \in \overline{A}[{\pi}^{-1}] \cap A^+.
$$
Since $\overline{A}$ is integrally closed in the field of fractions, we have $b' \in \overline{A}$ and $x \in \pi \overline{A}$.
\end{proof}

The regular local ring constructed in the following lemma is called a \textit{small Fontaine ring} in \cite{Sh11}. We need it for a ring extension $R \subset T$, where $T$ is possibly smaller than its absolute integral closure.

\begin{lemma}
\label{SmallFontaine}
Let $(R,\fm)$ be a complete Noetherian local domain of dimension $d$ and mixed characteristic $p>0$ and let $p,x_2,\ldots,x_d$ be a system of parameters. Let us choose compatible systems of elements:
$$
\{p^{\frac{1}{p^n}}\}_{n  \ge 0},\{x_2^{\frac{1}{p^n}}\}_{n  \ge 0},\ldots,\{x_d^{\frac{1}{p^n}}\}_{n  \ge 0}
$$
in $R^+$ and assume that $T$ is an integral pre-perfectoid domain such that $R \subset T \subset R^+$ and $\{p^{\frac{1}{p^n}}\}_{n  \ge 0},\{x_2^{\frac{1}{p^n}}\}_{n  \ge 0},\ldots,\{x_d^{\frac{1}{p^n}}\}_{n  \ge 0}$ are contained in $T$. Then there exist systems of elements $p^\flat, x_2^\flat,\ldots,x_d^\flat \in T^\flat$ associated to $p,x_2,\ldots,x_d$, together with a ring homomorphism:
$$
\psi:\mathbb{F}_p[[p^\flat,x_2^\flat,\ldots,x_d^\flat]] \to T^\flat
$$
such that $\mathbb{F}_p[[p^\flat,x_2^\flat,\ldots,x_d^\flat]]$ is a complete regular Noetherian local ring of dimension $d$ and characteristic $p>0$. Moreover, we have the equalities: $\psi(p^\flat)=p^\flat,\psi(x_2^\flat)=x_2^\flat,\ldots,\psi(x_d^\flat)=x_d^\flat$.
\end{lemma}

\begin{proof}
Choose a module-finite extension $V[[x_2,\ldots,x_d]] \hookrightarrow R$, where $V$ is an unramified complete discrete valuation ring. Moreover, we have
$$
W(\mathbb{F}_p)[[x_2,\ldots,x_d]] \hookrightarrow V[[x_2,\ldots,x_d]] 
\hookrightarrow R,
$$
where the first map is not necessarily module-finite. Now we want to apply the construction of Fontaine rings to the map $A:=W(\mathbb{F}_p)[[x_2,\ldots,x_d]] \hookrightarrow T$. There is a tower of of module-finite extensions of regular local contained in $T$:
$$
A=A_0 \hookrightarrow A_1 \hookrightarrow \cdots \hookrightarrow A_n 
\hookrightarrow \cdots
$$
such that $A_n:=A[p^{\frac{1}{p^n}},x_2^{\frac{1}{p^n}},\ldots,x_d^{\frac{1}{p^n}}]$. Since the residue field of $A_n$ is obviously perfect for $n \ge 0$, the Frobenius endomorphism $A_{n+1}/pA_{n+1} \to A_{n+1}/pA_{n+1}$ surjects onto the subring $A_n/pA_n$ of $A_{n+1}/pA_{n+1}$. We have a commutative diagram:
$$
\begin{CD}
A_{n+1}/pA_{n+1} @>>> T/pT \\
@VF_{n+1} VV @VF VV \\
A_n/pA_n @>>> T/pT \\
\end{CD}
$$
where the horizontal maps are induced by an inclusion $A_n \hookrightarrow T$, and the vertical maps $F_{n+1}$ are induced by the Frobenius on $A_{n+1}/pA_{n+1}$. By \cite[Proposition 4.5]{Sh11}, we find that the limit of the inverse system $\{A_n/pA_n;F_n \}_{n \ge 0}$ is isomorphic to the regular local ring $\mathbb{F}_p[[p^\flat,x_2^\flat,\ldots,x_d^\flat]]$, where the sequence $p^\flat,x_2^\flat,\ldots,x_d^\flat$ is determined by $$
\{p^{\frac{1}{p^n}}\}_{n  \ge 0},\{x_2^{\frac{1}{p^n}}\}_{n  \ge 0},\ldots,\{x_d^{\frac{1}{p^n}}\}_{n  \ge 0}.
$$
Taking the inverse limits to the above commutative diagram, we get a ring homomorphism:
$$
\mathbb{F}_p[[p^\flat,x_2^\flat,\ldots,x_d^\flat]] \to T^\flat,
$$
as desired.
\end{proof}

The following definition is taken from \cite{GR03} under some restrictions.

\begin{definition}
Let $T$ be a $p$-torsion free algebra and let $\pi \in T$ be a nonzero divisor and assume that $T$ admits a compatible system of elements $\{\pi_n\}_{n \ge 0}$ such that $\pi:=\pi_0$. Let us write $\pi:=\pi^{\frac{1}{p^n}} \in T$ for clarity. Then a $T$-module $M$ is said to be \textit{almost zero with respect to $\pi^{\frac{1}{p^{\infty}}}$}, if $\pi^{\frac{1}{p^n}} \cdot M=0$ for all $n>0$. A homomorphism of $T$-modules $f:M \to N$ is said to be \textit{almost isomorphic}, if $\ker(f)$ and $\coker(f)$ are almost zero with respect to $\pi^{\frac{1}{p^{\infty}}}$. In this case, we denote by $M \approx N$. 
\end{definition}

We shall say that the pair $(T,\pi^{\frac{1}{p^{\infty}}}T)$ is a \textit{basic setup} and keep in mind that various notions and definitions appearing in almost ring theory is dependent on a choice of such a basic setup. Let us put $I:=\pi^{\frac{1}{p^{\infty}}}T$. Then it is obvious that $I^2=I$ and $I$ is a flat $T$-module, so that results and techniques of \cite{GR03} and \cite{GR} can be applied. We refer the reader to \cite[Definition 4.1.1]{An2} and \cite{R10} for the definition of almost Cohen-Macaulay algebras.

\begin{definition}
\label{DefAlmost}
Let $(R,\fm)$ be a Noetherian local ring of dimension $d>0$ and mixed characteristic $p>0$, and let $(T,\pi^{\frac{1}{p^{\infty}}}T)$ be a basic setup equipped with an $R$-algebra structure. Then $T$ is called \textit{almost Cohen-Macaulay} (or \textit{$(\pi^{\frac{1}{p^\infty}})$-almost Cohen-Macaulay}), if $T/\fm T$ is not almost zero with respect to $\pi^{\frac{1}{p^{\infty}}}$, and there exists a system of parameters $x_1,\ldots,x_d$ of $R$ such that
$$
\pi^{\frac{1}{p^n}} \cdot \frac{\big((x_1,\ldots,x_i):_T x_{i+1}\big)}{(x_1,\ldots,x_i)}=0
$$
for all $n>0$ and $i=0,\ldots,d-1$.
\end{definition}

\begin{lemma}
\label{notalmostzero}
Let $(R,\fm)$ be a Noetherian local ring and let $(B,\pi^{\frac{1}{p^{\infty}}}B)$ be a basic setup over $R$, where $\pi$ is part of a system of parameters of $R$. Assume that $\{B_{\lambda}\}_{\lambda \in \Lambda}$ is a filtered direct system of $B$-algebras that are seeds over $R$. Let $B_{\infty}$ be its direct limit. If $x_1,\ldots,x_d$ is a system of parameters of $R$, then 
$$
\frac{B_{\infty}}{(x_1,\ldots,x_d)B_{\infty}}
$$
is not an almost zero $B$-module.
\end{lemma}

\begin{proof}
By \cite[Lemma 3.2]{Di07}, it follows that $B_{\infty}$ is a seed over $R$ by assumption. Then there is a map $B_{\infty} \to T$ such that $T$ is a big Cohen-Macaulay $R$-algebra. By Proposition \ref{Strooker}, we may assume that $T$ is an $\fm$-adically complete and separated big Cohen-Macaulay algebra. In particular, since $\pi$ is part of a system of parameters of $R$, it follows that $\pi$ is a nonzero divisor in $T$. Assume to the contrary that $\pi^{\frac{1}{p^{\infty}}}B_{\infty} \subset (x_1,\ldots,x_d)B_{\infty}$. That is, we get
$$
\pi^{\frac{1}{p^n}}=\sum_{i=1}^d a^{(n)}_i x_i~\mbox{with}~a^{(n)}_i \in B_{\infty}~\mbox{for every}~n>0.
$$
By mapping this relation to $T$, we get $\pi^{\frac{1}{p^n}} \in (x_1,\ldots,x_d)T$. In other words,
$$
\pi \in \bigcap_{n>0} (x_1,\ldots,x_d)^{p^n}T=(0),
$$
because $T$ is $\fm$-adically separated. However, this is a contradiction.
\end{proof}

Let $(R,\fm)$ be a complete Noetherian local normal domain of mixed characteristic with perfect residue field $k$ and choose a module-finite extension: $A:=W(k)[[x_2,\ldots,x_d]] \hookrightarrow R$ and fix compatible systems of elements in $R^+$:
$$
\{p^{\frac{1}{p^n}}\}_{n  \ge 0},\{x_2^{\frac{1}{p^n}}\}_{n  \ge 0},\ldots,\{x_d^{\frac{1}{p^n}}\}_{n  \ge 0}.
$$

\begin{definition}
\label{Faltings}
Let the notation be as above. We set
$$
A_{\infty}:=\bigcup_{n \ge 0} A[p^{\frac{1}{p^n}},x_2^{\frac{1}{p^n}},\ldots,x_d^{\frac{1}{p^n}}].
$$
and define $R_{\infty}$ to be the normalization of the join $R[A_{\infty}]$ in its field of fractions.
\end{definition}

The ring $A_{\infty}$ has been defined previously in \cite{Sh07} and \cite{Sh16} (in the latter paper, the ring $A_{\infty}$ was used to prove the existence of a big Cohen-Macaulay algebra under a certain condition, using Faltings' Almost Purity Theorem). It is an integral pre-perfectoid, henselian quasi-local normal domain.

\begin{lemma}
\label{Fontaine}
Let $A_{\infty}$ and $R_{\infty}$ be as in Definition \ref{Faltings}. Then $R_{\infty}$ is a henselian quasi-local normal domain and the ring extension $R \to R_{\infty}$ is ind-\'etale after an inversion of $px_2\cdots x_d$.
\end{lemma}

\begin{proof}
Since $R_{\infty}$ is an integral extension domain of $A_{\infty}$, it is a henselian quasi-local normal domain. Next, we set $R_n:=R[p^{\frac{1}{p^n}},x_2^{\frac{1}{p^n}},\ldots,x_d^{\frac{1}{p^n}}]$. It is easily checked that
$$
R[\frac{1}{px_2\cdots x_d}] \to R_n[\frac{1}{px_2\cdots x_d}]
$$
is finite \'etale.\footnote{One uses the following fact: Let $R[a]$ be a simple ring extension of $R$ contained in $R^+$. Then there is a unique monic polynomial $f \in R[x]$ such that $R[x]/(f) \cong R[a]$ by the normality of $R$.} Taking the direct limit, $R[\frac{1}{px_2\cdots x_d}] \to \varinjlim_n R_n[\frac{1}{px_2\cdots x_d}]$ is ind-\'etale and thus, $\varinjlim_n R_n[\frac{1}{px_2\cdots x_d}]$ is a normal domain. On the other hand, $R_{\infty}[\frac{1}{px_2\cdots x_d}]$ is a normal domain that is integral over $R[\frac{1}{px_2\cdots x_d}]$ and it has the same field of fractions as that of $\varinjlim_n R_n[\frac{1}{px_2\cdots x_d}]$. So we have
$$
\varinjlim_n R_n[\frac{1}{px_2\cdots x_d}]=R_{\infty}[\frac{1}{px_2\cdots x_d}],
$$
which proves the lemma.
\end{proof}

We should notice that $R_{\infty}$ is \textit{not} integral pre-perfectoid in general. The following remarkable result was established by Y. Andr\'e as a consequence of using his Perfectoid Abhyankar's Lemma. We state it in the form we need.

\begin{theorem}[Y. Andr\'e]
\label{almostCM}
Let $(R,\fm)$ be a complete Noetherian local normal domain of mixed characteristic $p>0$ with perfect residue field $k$ and let $\{R_{\lambda}\}_{\lambda \in \Lambda}$ be a filtered direct system of local normal domains that are injective, module-finite over $R$. Choose a module-finite extension
$$
A:=W(k)[[x_2,\ldots,x_d]] \hookrightarrow R.
$$
Then there exists an element $g \in A \setminus pA$ such that
$$
A[\frac{1}{pg}] \to R[\frac{1}{pg}]
$$
is \'etale. Assume that $A[\frac{1}{pg}] \to R_{\lambda}[\frac{1}{pg}]$ is \'etale for every $\lambda \in \Lambda$. Then there exists a family of integral almost perfectoid, almost Cohen-Macaulay $R$-algebras $\{T_{\lambda}\}_{\lambda \in \Lambda}$ with respect to $\pi^{\frac{1}{p^\infty}}$ such that $\pi:=pg$ and that fills in the following commutative diagram of $R$-algebras:
$$
\begin{CD}
T_{\lambda_1} @>>> T_{\lambda_2} \\
@AAA @AAA \\
R_{\lambda_1} @>>> R_{\lambda_2} \\
\end{CD}
$$
for every $\lambda_1 \le \lambda_2$ in $\Lambda$, where $T_{\lambda_1} \to T_{\lambda_2}$ is integral.
\end{theorem}

\begin{proof}
The statement asserts that the family $\{T_{\lambda}\}_{\lambda \in \Lambda}$ forms a filtered direct system and one can take its direct limit, which plays a role in the subsequent theorem.

Since $A \to R$ is generically \'etale, we can find $g \in A \setminus pA$ such that
$$
A[\frac{1}{pg}] \to R[\frac{1}{pg}]
$$
is finite \'etale. The localization $A[\frac{1}{pg}] \to R_{\lambda}[\frac{1}{pg}]$ is \'etale by assumption, so we can find an integral almost perfectoid, almost Cohen-Macaulay $R_{\lambda}$-algebra $T_{\lambda}$ by \cite[(18) in \S~4.2]{An2} with respect to $\pi^{\frac{1}{p^{\infty}}}$ and $\pi:=pg$. Notice that $T_{\lambda}$ contains a compatible system of $p^n$-th roots of $p$ (resp. $g$) for all $n>0$. It is necessary to look at the intrinsic structure of $T_{\lambda}$. 
Let $A \hookrightarrow R \hookrightarrow B:=R_{\lambda}$ be a composite of module-finite extensions for $\lambda \in \Lambda$, and let us put $T_{\lambda}:=\mathcal{B}^{\circ}$, where $\mathcal{B}^{\circ}$ appears as in \cite[\S~3.2]{An2} associated to the module-finite extension $A \to B$. Suppose that $R \hookrightarrow R_{\lambda_1} \hookrightarrow R_{\lambda_2}$ is a composite of module-finite extensions for $\lambda_1,\lambda_2 \in \Lambda$ with $\lambda_1 \le \lambda_2$. Then we obtain a commutative diagram
$$
\begin{CD}
T_{\lambda_1} @>>> T_{\lambda_2} \\
@AAA @AAA \\
R_{\lambda_1} @>>> R_{\lambda_2} \\
\end{CD}
$$
where $T_{\lambda_1} \to T_{\lambda_2}$ is an integral ring map, as one observes from the construction given in \cite[\S~3.2]{An2}. This completes the proof of the theorem.
\end{proof}

\section{Main theorems}

Let us establish the following theorem via Theorem \ref{almostCM}.

\begin{theorem}
\label{seedFontaine}
Let $(R,\fm)$ be a complete Noetherian local normal domain of dimension $d$ and mixed characteristic $p>0$ with perfect residue field $k$. Choose a system of parameters $p,x_2,\ldots,x_d$ of $R$ and $R_{\infty}$ as in Definition \ref{Faltings}.
Let
$$
R_{\infty}[\frac{1}{p}] \hookrightarrow R_{\infty}[\frac{1}{p}]^{\rm{\acute{e}t}}
$$
be the maximal \'etale extension inside $R^+[\frac{1}{p}]$. Let $\overline{R}$ denote the integral closure of $R$ in $R_{\infty}[\frac{1}{p}]^{\rm{\acute{e}t}}$. Then the following assertions hold:
\begin{enumerate}
\item
$\overline{R}$ is an integral pre-perfectoid normal domain.

\item
There exists a perfect $\mathbb{F}_p$-algebra $\mathscr{B}({\overline{R}^\flat})$, together with a ring homomorphism $\overline{R}^\flat \to \mathscr{B}({\overline{R}^\flat})$ such that $\mathscr{B}({\overline{R}^\flat})$ is a balanced big Cohen-Macaulay $\mathbb{F}_p[[p^\flat,x_2^\flat,\ldots,x_d^\flat]]$-algebra (see Lemma \ref{SmallFontaine} for the notation).
\end{enumerate}
\end{theorem}

\begin{proof}
The assertion $\rm (1)$ is due to Lemma \ref{semiperfect}.

We prove the assertion $\rm (2)$. Take a module-finite extension $A:=W(k)[[x_2,\ldots,x_d]] \hookrightarrow R$. Since $A \to R$ is generically \'etale, there is a nonzero element $g \in A \setminus pA$ such that
$$
A[\frac{1}{pg}] \hookrightarrow R[\frac{1}{pg}]
$$
is finite \'etale. It follows from Lemma \ref{Fontaine} that
$$
R[\frac{1}{px_2 \cdots x_d}] \hookrightarrow R_{\infty}[\frac{1}{px_2 \cdots x_d}]
$$
is ind-\'etale. Let $\overline{R}$ be as in the hypothesis of the theorem. Then the integral extension
$$
A[\frac{1}{pgx_2\cdots x_d}] \hookrightarrow \overline{R}[\frac{1}{pgx_2\cdots x_d}]
$$
is ind-\'etale. For simplicity, we write $g$ for $gx_2\cdots x_d$. By the definition of ind-\'etale extensions, one can present $\overline{R}$ as the colimit $\varinjlim_{\lambda \in \Lambda} R_{\lambda}$ such that $R \hookrightarrow R_\lambda$ is module-finite, $R_\lambda$ is a Noetherian normal domain contained in $\overline{R}$, and $A[\frac{1}{pg}] \to R_\lambda[\frac{1}{pg}]$ is a finite \'etale 
extension for all $\lambda \in \Lambda$. Theorem \ref{almostCM} then yields a commutative diagram
$$
\begin{CD}
T_{\lambda_1} @>>> T_{\lambda_2} \\
@AAA @AAA \\
R_{\lambda_1} @>>> R_{\lambda_2} \\
\end{CD}
$$
for any pair $\lambda_1 \le \lambda_2$ in $\Lambda$. Hence, $\{T_{\lambda}\}_{\lambda \in \Lambda}$ forms a filtered direct system and denote by $T_{\infty}$ its direct limit. Notice that the $p$-adic completed algebra $T_{\infty}^{\wedge}$ is an integral almost perfectoid algebra. Now we claim that 
\begin{equation}
\label{notcontain}
\frac{T_{\infty}}{(p,x_2,\ldots,x_d)T_{\infty}}~\mbox{is not almost zero with respect to}~\pi^{\frac{1}{p^{\infty}}}.
\end{equation}
For the proof of this claim, we may proceed as in the discussions given in \cite[\S~4.2]{An2}: It is shown that $T_{\lambda}$ maps to a big Cohen-Macaulay $R$-algebra for $\lambda \in \Lambda$ by \cite[Proposition 4.1.2]{An2}. That is, $T_{\lambda}$ is a seed over $R$ for all $\lambda \in \Lambda$. Therefore, 
$(\ref{notcontain})$ follows by Lemma \ref{notalmostzero}. It follows that the $p$-adic completion $T_{\infty}^{\wedge}$ is an integral almost perfectoid, almost Cohen-Macaulay algebra and we have $T_{\infty}^\flat \cong (T_{\infty}^\wedge)^{\flat}$.

On the other hand, $\overline{R}$ is an integral pre-perfectoid normal domain, which implies that the natural map $\overline{R}^\flat \to \overline{R}/p\overline{R}$ is surjective. Now we get the homomorphism of Fontaine rings:
$$
\overline{R}^\flat \to T_{\infty}^\flat.
$$
Let $p^\flat,x_2^\flat,\ldots,x_d^\flat \in \overline{R}^\flat$ be a sequence of elements as in the theorem. Then $\overline{R}^\flat$ contains a complete regular local ring $\mathbb{F}_p[[p^\flat,x_2^\flat,\ldots,x_d^\flat]]$. Now it suffices to show that $T_{\infty}^\flat$ is a seed over $\mathbb{F}_p[[p^\flat,x_2^\flat,\ldots,x_d^\flat]]$. There is an injection $T_{\infty}^\flat/p^\flat T_{\infty}^\flat \hookrightarrow T_{\infty}/pT_{\infty}$ whose cokernel is almost zero with respect to $\pi^{\frac{1}{p^{\infty}}}$ in view of \cite[Proposition 3.5.4]{An1}, and we find that
$$
(\pi^\flat)^{\frac{1}{p^n}} \cdot \dfrac{\big((p^\flat,x_2^\flat,\ldots,x_i^\flat):_{\mathcal{T}_{\infty}^\flat} x_{i+1}^\flat \big)}{(p^\flat,x_2^\flat,\ldots,x_i^\flat)}=0
$$
for all integers $n >0$ and $0 \le i \le \dim R-1$. Moreover, the almost isomorphism
$$
\frac{T_{\infty}^\flat}{p^\flat T_{\infty}^\flat} \approx \frac{T_{\infty}}{pT_{\infty}}
$$
induces an almost isomorphism
$$
\frac{T_{\infty}^\flat}{(p^\flat,x_2^\flat,\ldots,x_d^\flat)T_{\infty}^\flat} \approx \frac{T_{\infty}}{(p,x_2,\ldots,x_d)T_{\infty}}.
$$
Then it follows from $(\ref{notcontain})$ that $T_{\infty}^\flat/(p^\flat,x_2^\flat,\ldots,x_d^\flat)T_{\infty}^\flat$ is not almost zero with respect to $(\pi^\flat)^{\frac{1}{p^{\infty}}}$. Thus, we find that $T_{\infty}^\flat$ is a seed over $\mathbb{F}_p[[p^\flat,x_2^\flat,\ldots,x_d^\flat]]$ in view of \cite[Proposition 4.1.2]{An2}. By applying Theorem \ref{seed}, we find that $T_{\infty}^\flat$ maps to a perfect balanced big Cohen-Macaulay $\mathbb{F}_p[[p^\flat,x_2^\flat,\ldots,x_d^\flat]]$-algebra $\mathscr{B}(\overline{R}^\flat)$. This completes the proof of the theorem.
\end{proof}

\begin{corollary}
\label{AbBig}
Let $(R,\fm)$ be a complete Noetherian local normal domain of mixed characteristic with perfect residue field and let $R^+$ be its absolute integral closure. Then the Fontaine ring $(R^+)^\flat$ maps to a perfect big Cohen-Macaulay $\mathbb{F}_p[[p^\flat,x_2^\flat,\ldots,x_d^\flat]]$-algebra.
\end{corollary}

\begin{proof}
Let $R \to S$ be a module-finite extension of normal domains. Fix a module-finite extension:
$$
A:=W(k)[[x_2,\ldots,x_d]] \hookrightarrow R.
$$
As in Theorem \ref{seedFontaine}, we define an integral pre-perfectoid normal
domain $\overline{S}$ as the integral closure $S$ in $S_{\infty}[\frac{1}{p}]^{\rm{\acute{e}t}}$. Notice that $R \to R^+$ is an integral extension and $R^+$ is normal. So we find that $R^+$ is described as a filtered colimit of all module-finite normal $R$-algebras that are contained in $R^+$. Then we have $R^+=\varinjlim_{\{R \to S \}} S$, where $R \to S$ ranges over all module-finite normal $R$-algebras contained in $R^+$. Since we have $S \hookrightarrow \overline{S} \hookrightarrow R^+$, it follows that $R^+=\varinjlim_{\{R \to \overline{S}\}} \overline{S}$. Consider the direct system of Fontaine rings $\{\overline{S}^\flat\}_{\{R \to \overline{S}\}}$ which is induced by applying the construction of Fontaine rings to $\{\overline{S}\}_{\{R \to \overline{S}\}}$. Moreover by Theorem \ref{seedFontaine}, $\{\overline{S}^\flat\}_{\{R \to \overline{S}\}}$ forms a direct system of seeds over $\mathbb{F}_p[[p^\flat,x_2^\flat,\ldots,x_d^\flat]]$. Hence, its direct limit
$$
D((R^+)^\flat):=\varinjlim_{\{R \to \overline{S}\}} \overline{S}^\flat
$$
is a seed over $\mathbb{F}_p[[p^\flat,x_2^\flat,\ldots,x_d^\flat]]$ in view of \cite[Lemma 3.2]{Di07}. So it follows from Theorem \ref{seed} that there is a perfect big Cohen-Macaulay $\mathbb{F}_p[[p^\flat,x_2^\flat,\ldots,x_d^\flat]]$-algebra $\mathscr{B}(D((R^+)^\flat))$.

Next, we prove that the $p^\flat$-adic completion of $D((R^+)^\flat)$ is isomorphic to $(R^+)^\flat$.
There are a natural homomorphism $\varinjlim_{\{R \to \overline{S}\}} \overline{S}^\flat \to (R^+)^\flat$ and
an isomorphism $\varinjlim_{\{R \to \overline{S}\}} \overline{S}/p\overline{S} \cong R^+/pR^+$. As $\overline{S}^\flat$ surjects onto $\overline{S}/p\overline{S}$ with kernel equal to $p^\flat \overline{R}^\flat$, and similarly $(R^+)^\flat$ surjects onto $R^+/pR^+$ with kernel equal to $p^\flat (R^+)^\flat$ by
Proposition \ref{FontaineRing}, it follows that the induced homomorphism
$$
\frac{D((R^+)^\flat)}{p^\flat D((R^+)^\flat)} \to \frac{(R^+)^\flat}{p^\flat (R^+)^\flat}
$$
is an isomorphism. As $D((R^+)^\flat)$ and $(R^+)^\flat$ are perfect $\mathbb{F}_p$-algebras, $k$ times iteration of the Frobenius yields isomorphisms:
$$
\frac{D((R^+)^\flat)}{p^\flat D((R^+)^\flat)} \cong \frac{D((R^+)^\flat)}{(p^\flat)^{p^k} D((R^+)^\flat)}~\mbox{and}~\frac{(R^+)^\flat}{p^\flat (R^+)^\flat} \cong \frac{(R^+)^\flat}{(p^\flat)^{p^k} (R^+)^\flat}.
$$
These facts imply that the $p^\flat$-adic completion of $D((R^+)^\flat)$ will be isomorphic to $(R^+)^\flat$. Taking the $p^\flat$-adic completion of the map $D((R^+)^\flat) \to \mathscr{B}(D((R^+)^\flat))$, we construct a perfect big Cohen-Macaulay algebra over $(R^+)^\flat$.\footnote{It is shown in \cite[Lemma 3.6]{Di07} that the ideal-adic completion of a perfect $\mathbb{F}_p$-algebra remains perfect.}
\end{proof}

Now we can prove the following theorem.

\begin{theorem}
\label{integralperfect}
Let $(R,\fm)$ be a Noetherian local domain of mixed characteristic. Then there exists an $R$-algebra $T$ such that $T$ is an integral perfectoid big Cohen-Macaulay $R$-algebra. 

Moreover, we have the following assertions:
\begin{enumerate}
\item
Assume that $R$ is a complete Noetherian local domain of mixed characteristic with perfect residue field. Then there exists an integral perfectoid big Cohen-Macaulay $R$-algebra $T$ with the property that $R \to T$ factors as $R \to S \to T$, such that $S$ is an integral pre-perfectoid normal domain that is integral over $R$ and $R[\frac{1}{pg}] \to S[\frac{1}{pg}]$ is ind-\'etale for some nonzero element $g \in R$.

\item
Assume that $R$ is a complete Noetherian local domain of mixed characteristic. Let $B$ be an integral almost perfectoid, almost Cohen-Macaulay $R$-algebra such that $R \to B$ factors as $R \to S \to B$ and $S$ is an integral perfectoid algebra containing compatible systems of elements: $
\{p^{\frac{1}{p^n}}\}_{n  \ge 0},\{x_2^{\frac{1}{p^n}}\}_{n  \ge 0},\ldots,\{x_d^{\frac{1}{p^n}}\}_{n  \ge 0}$ for a system of parameters $p,x_2,\ldots,x_d$ of $R$. Then there is a ring homomorphism $S \to T$ such that $T$ is an integral perfectoid big Cohen-Macaulay $R$-algebra.
\end{enumerate}
\end{theorem}

\begin{proof}
We first establish the existence of an integral perfectoid big Cohen-Macaulay $R$-algebra. Let us enlarge the residue field $R/\fm$ to the perfect residue field, so we get a flat local homomorphism $(R,\fm) \to (R',\fm')$ of same Krull dimension such that $R'/\fm'$ is a perfect field. After killing some minimal prime ideal of the $\fm'$-adic completion of $R'$, take its normalization in the field of fractions. So we may assume that $R$ is a complete Noetherian local normal domain of mixed characteristic with perfect residue field without loss of generality. Let $S:=\overline{R}$ be as in Theorem \ref{seedFontaine} $\rm(1)$. Then we get an $\overline{R}^\flat$-algebra $\mathscr{B}(\overline{R}^\flat)$ as constructed in Theorem \ref{seedFontaine} $\rm(2)$. Applying Proposition \ref{Teichmuller}, we get ring homomorphisms:
$$
\overline{R} \to \overline{R}^{\wedge} \cong \frac{W(\overline{R}^\flat)}{\vartheta \cdot W(\overline{R}^\flat)} \to \frac{W(\mathscr{B}(\overline{R}^\flat))}{\vartheta \cdot W(\mathscr{B}(\overline{R}^\flat))}.
$$

Let us put
$$
T:=\frac{W(\mathscr{B}(\overline{R}^\flat))}{\vartheta \cdot W(\mathscr{B}(\overline{R}^\flat))}.
$$
We show that $T$ is $p$-adically complete. To this aim, let us consider the exact sequence
\begin{equation}
\label{exact}
0 \to W(\mathscr{B}(\overline{R}^\flat)) \xrightarrow{\times \vartheta} W(\mathscr{B}(\overline{R}^\flat)) \to T \to 0.
\end{equation}
Since $T$ is a $\mathbb{Z}_p$-flat algebra, applying $(-) \otimes_{\mathbb{Z}_p} \mathbb{Z}_p/p^n \mathbb{Z}_p$ to the sequence, we get a short exact sequence:
$$
0 \to \frac{W(\mathscr{B}(\overline{R}^\flat))}{p^n W(\mathscr{B}(\overline{R}^\flat))} \xrightarrow{\times \vartheta} \frac{W(\mathscr{B}(\overline{R}^\flat))}{p^n W(\mathscr{B}(\overline{R}^\flat))} \to \frac{T}{p^n T} \to 0.
$$
Then taking the inverse limit with respect to $n>0$, together with the snake lemma, it follows that $T$ is $p$-adically complete. Next, we show that the Frobenius map is surjective on $T/pT$. This follows from the isomorphisms:
$$
\frac{T}{pT} \cong \frac{W(\mathscr{B}(\overline{R}^\flat))}{(p^\flat,p)W(\mathscr{B}(\overline{R}^\flat))} \cong
\frac{\mathscr{B}(\overline{R}^\flat)}{p^\flat \mathscr{B}(\overline{R}^\flat)},
$$
together with the fact that $\mathscr{B}(\overline{R}^\flat)$ is a perfect $\mathbb{F}_p$-algebra. Moreover, let $\varpi$ be the image of $(p^\flat)^{\frac{1}{p}}$ via the composite maps $\mathscr{B}(\overline{R}^\flat) \to W(\mathscr{B}(\overline{R}^\flat)) \to T$, where the first map is just the multiplicative map $\theta_{\mathscr{B}(\overline{R}^\flat)}$. Then we have $\varpi^p=p$ and the Frobenius bijection on $\mathscr{B}(\overline{R}^\flat)$ induces an isomorphism:
$$
 \frac{T}{\varpi T} \cong \frac{\mathscr{B}(\overline{R}^\flat)}{(p^\flat)^{\frac{1}{p}} \mathscr{B}(\overline{R}^\flat)}
 \xrightarrow{F} \frac{\mathscr{B}(\overline{R}^\flat)}{p^\flat \mathscr{B}(\overline{R}^\flat)}
\cong \frac{T}{pT},
$$
which proves that $T$ is an integral perfectoid $R$-algebra. Finally, in view of Theorem \ref{seedFontaine} $\rm(2)$, it follows that $p,x_2,\ldots,x_d$ is a regular sequence on $T$. Hence we complete the proof of the existence of an integral perfectoid big Cohen-Macaulay $R$-algebra. 

The assertion $\rm(1)$ is clear from the construction given in the above proof and in Theorem \ref{seedFontaine}. Notice that $R[\frac{1}{pg}]$ is a normal domain as in the construction of the proof of Theorem \ref{seedFontaine}, even though $R$ is not assumed to be normal. For the assertion $\rm(2)$, one concludes from the proof of Theorem \ref{seedFontaine} that $B^\flat$ maps to a perfect big Cohen-Macaulay $\mathbb{F}_p[[p^\flat,x_2^\flat,\ldots,x_d^\flat]]$-algebra $\mathscr{B}(B^\flat)$. Applying the construction of Witt vectors as above to $S^\flat \to \mathscr{B}(B^\flat)$, we get an integral perfectoid big Cohen-Macaulay $S$-algebra, as desired. We now complete the proof.
\end{proof}

\begin{remark}
In the second statement $\rm(2)$ of Theorem \ref{integralperfect}, one can indeed prove that there is a ring homomorphism $B^{\natural} \to T$, where $B^{\natural}$ denotes the untilt of the tilt of $B$ (see \cite{KL15} and \cite{Sch12} for the tilting and tilting correspondence). That is,
$$
B^{\natural}:=\frac{W(B^\flat)}{\vartheta \cdot  W(B^\flat)}.
$$
However, it is not necessarily true that $B^{\natural}$ is isomorphic to $B$.
\end{remark}

One can indeed prove the existence of an integral perfectoid big Cohen-Macaulay algebra over an absolutely integrally closed domain.

\begin{corollary}
\label{absoluteBig}
Let $(R,\fm)$ be a Noetherian local domain of mixed characteristic and let $R^+$ be its absolute integral closure. Then $R^+$ maps to an integral perfectoid big Cohen-Macaulay $R$-algebra.
\end{corollary}

\begin{proof}
Indeed, we can construct a flat local extension of $R$, so that we may assume that $R$ is a complete Noetherian local normal domain with perfect residue field. The corollary then follows by combining Corollary \ref{AbBig}, together with the proof of Theorem \ref{integralperfect} by applying to $R^+$.
\end{proof}

\begin{remark}
\label{completion}
Let $T$ be an integral perfectoid algebra and let $J \subset T$ be a finitely generated ideal such that $p \in J$. It can be proved that the $J$-adic completion of $T$ remains integral perfectoid (see \cite[Proposition 3.6.19]{KL15} for the proof of this fact). By taking the $\fm$-adic completion of the algebra obtained in Corollary \ref{absoluteBig}, $R^+$ can be mapped into an integral perfectoid \textit{balanced} big Cohen-Macaulay algebra.
\end{remark}

From the discussions we have made so far, one easily deduces the following.

\begin{proposition}
\label{corFontaine}
Let $(R,\fm)$ be a complete Noetherian local domain of mixed characteristic with a fixed system of parameters $\mathbf{x}:=p,\ldots,x_d$. Fix an integral perfectoid $R$-algebra $S$ which contains a compatible system of elements:
$$
\{p^{\frac{1}{p^n}}\}_{n  \ge 0},\{x_2^{\frac{1}{p^n}}\}_{n  \ge 0},\ldots,\{x_d^{\frac{1}{p^n}}\}_{n  \ge 0}.
$$
Let $\mathbf{x}^\flat:=p^\flat,\ldots,x_d^\flat$ be lifts of $\mathbf{x}:=p,\ldots,x_d$ to $S^\flat$. Then there is an equivalence of categories:
$$
\biggl\{\mbox{Integral perfectoid big Cohen-Macaulay}~S\mbox{-algebras with respect to}~\mathbf{x}\biggl\}
$$
$$
\to \biggl\{p^\flat\mbox{-adically complete perfect big Cohen-Macaulay}~S^\flat\mbox{-algebras with respect to}~\mathbf{x}^\flat\biggl\},
$$
which is defined by $T \mapsto T^\flat$.
\end{proposition}

\begin{proof}
Let $T$ be a $p^\flat$-adically complete perfect big Cohen-Macaulay $S^\flat$-algebra with respect to the sequence $\mathbf{x}^\flat$. Then as in the proof of Theorem \ref{integralperfect}, the quasi-inverse of the above functor is given by
$$
T \mapsto \frac{W(T)}{\vartheta \cdot W(T)}
$$
and the target ring is an integral perfectoid big Cohen-Macaulay $S$-algebra.
\end{proof}

\section{Comments and open problems}

After this paper was written, there has been remarkable progresses on the homological conjectures in mixed characteristic. Among them, we want to remark that Andr\'e established the following theorem in \cite{An3}, using the results of this paper.\footnote{Quite recently, O. Gabber also announced a similar result, using his weak local uniformization theorem and ultra-products of rings.}

\begin{theorem}[Weak functoriality]
Let $(R,\fm) \to (S,\fn)$ be a local homomorphism of complete Noetherian local domains such that $R$ has mixed characteristic $p>0$. Then there exists a commutative diagram
$$
\begin{CD}
R @>>> S \\
@VVV @VVV \\
R^+ @>>> S^+ \\
@VVV @VVV \\
\mathscr{B}(R) @>>> \mathscr{B}(S) \\
\end{CD}
$$
such that $\mathscr{B}(R)$ is an integral perfectoid big Cohen-Macaulay $R$-algebra, and $\mathscr{B}(S)$ is an integral perfectoid big Cohen-Macaulay $S$-algebra (resp. perfect big Cohen-Macaulay $S$-algebra), if $S$ has mixed characteristic $p>0$ (resp. positive characteristic $p>0$).
\end{theorem}

Refer to the paper \cite{DRG} for a recent progress made on the weakly functorial big Cohen-Macaulay algebras in the equal characteristic zero case, using ultra-products. In \cite{HM17}, Heitmann and Ma proved that if $R$ is a regular local ring of mixed characteristic, then the direct summands of $R$ are pseudo-rational; in particular, they are Cohen-Macaulay. In \cite{MS18}, Ma and Schwede established a comparison theorem of symbolic and ordinary powers of ideals in an excellent regular ring. In the case when the ring has equal characteristic, the same result had been obtained earlier by Ein-Lazarsfeld-Smith and Hochster-Huneke (see \cite{MS18} for a brief history on these results). We also mention the work of Bhatt, Iyenagr and Ma \cite{BIM18} on the mixed characteristic analogue of Kunz's theorem on the regularity of Noetherian $\mathbb{F}_p$-algebras via the flatness of the Frobenius map.

We propose the following problem, which seems to be an important step to consider tight closure-like operations in mixed characteristic. Some relevant results appear in a recent paper \cite{HM18}.

\begin{problem}
\label{Pro2}
Let $(R,\fm)$ be a complete Noetherian local (normal) domain of mixed characteristic and assume that $A$ and $B$ are integral perfectoid big Cohen-Macaulay $R$-algebras. Then does $A \otimes_R B$ map to an integral perfectoid big Cohen-Macaulay $R$-algebra?
\end{problem}

The paper \cite{HM18} has partial results on the following problem.

\begin{problem}
Let $(R,\fm)$ be a complete Noetherian local domain of mixed characteristic and let $T$ denote the $p$-adic completion of the absolute integral closure $R^+$. Then is $T$ an almost Cohen-Macaulay algebra?
\end{problem}

We refer the interested reader to \cite{An4} and \cite{Sh18}, both of which give accounts for ideas surrounding Andr\'e's Perfectoid Abhyankar's Lemma and applications to the homological conjectures.

\begin{acknowledgement}
I am grateful to Bhargav Bhatt, Raymond Heitmann, Kiran Kedlaya, Linquan Ma, Kei Nakazato and Paul Roberts for useful comments. My special gratitude goes to Yves Andr\'e for his numerous and kind comments and for his inspiration. Finally, I am grateful to the referee for pointing out errors and providing remarks.
\end{acknowledgement}

\end{document}